\documentclass[11pt]{amsart}
\usepackage{amssymb}
\usepackage{amsmath}

\usepackage[plainpages]{hyperref}

\usepackage[all]{xy}

\CompileMatrices

\let\Bbb\mathbb

\def\>{\relax\ifmmode\mskip.666667\thinmuskip\relax\else\kern.111111em\fi}
\def\<{\relax\ifmmode\mskip-.333333\thinmuskip\relax\else\kern-.0555556em\fi}
\def\vsk#1>{\vskip#1\baselineskip}
\def\vv#1>{\vadjust{\vsk#1>}\ignorespaces}
\def\vvn#1>{\vadjust{\nobreak\vsk#1>\nobreak}\ignorespaces}

\let\Medskip\medskip
\def\medskip{\par\Medskip}
\let\Bigskip\bigskip
\def\bigskip{\par\Bigskip}

\let\Maketitle\maketitle
\def\maketitle{\Maketitle\thispagestyle{empty}\let\maketitle\empty}

\newtheorem{thm}{Theorem}[section]
\newtheorem{cor}[thm]{Corollary}
\newtheorem{lem}[thm]{Lemma}
\newtheorem{prop}[thm]{Proposition}

\newenvironment{proofof}[1]{
\noindent {\it Proof of #1:}}{\hfill \qed \medskip
}

\numberwithin{equation}{section}

\theoremstyle{definition}
\newtheorem{rem}[thm]{Remark}
\newtheorem{example}[thm]{Example}
\newtheorem{defn}[thm]{Definition}

\newtheorem{ack}{Acknowledgements}

\let\mc\mathcal
\let\nc\newcommand

\nc{\es}{\ensuremath}
\nc{\on}{\operatorname}
\nc{\codim}{\on{codim}}
\nc{\Z}{{\mathbb Z}}
\nc{\C}{\es{\mathbb C}}
\nc{\N}{{\mathbb N}}
\nc{\pone}{{\mathbb C}{\mathbb P}^1}
\renewcommand{\P}{\mathbb P}
\nc{\arr}{\rightarrow}
\nc{\larr}{\longrightarrow}
\nc{\al}{\alpha}
\nc{\W}{{\mc W}}
\nc{\la}{\lambda}
\nc{\su}{\widehat{{\mathfrak sl}}_2}
\nc{\g}{{\mathfrak g}}
\nc{\h}{{\mathfrak h}}
\nc{\m}{{\mathfrak m}}
\nc{\n}{{\mathfrak n}}
\nc{\Gm}{\Gamma}
\nc{\La}{\Lambda}
\nc{\gl}{\widehat{\mathfrak{gl}_2}}
\nc{\bi}{\bibitem}
\nc{\om}{\omega}
\nc{\Res}{\on{Res}}
\nc{\gm}{\gamma}
\nc{\Om}{\Omega}
\nc{\Hom}{\es{\on{Hom}}}
\nc{\A}{\es{\mc A}}
\nc{\bAA}{\es{\bar{\A}}}

\nc{\bA}{\es{\bar{A}}}
\nc{\bS}{\es{\bar{S}}}
\nc{\bH}{\bar{H}}
\renewcommand{\d}{{\rm d}}
\nc{\dA}{\d \A}
\nc{\bM}{\es{\bar{M}}}
\nc{\we}{\wedge}

\def\im{\on{im}}

\def\Ann{\on{Ann}}

\def\Res{\on{Res}}

\def\F{\es{\mc F}}
\def\bF{\es{\bar{\mc F}}}
\def\L{{\mc L}}

\let\geq\geqslant

\let\leq\leqslant

\nc{\gln}{\mathfrak{gl}_N}
\nc{\sln}{\mathfrak{sl}_N}

\def\beq{\begin{equation}}
\def\eeq{\end{equation}}
\def\be{\begin{equation*}}
\def\ee{\end{equation*}}

\nc{\bean}{\begin{eqnarray}}
\nc{\eean}{\end{eqnarray}}
\nc{\bea}{\begin{eqnarray*}}
\nc{\eea}{\end{eqnarray*}}
\nc{\bs}{\boldsymbol}
\nc{\Ref}[1]{{\rm(\ref{#1})}}

\nc{\kk}{\es{\Bbbk}}
\nc{\R}{\Bbb R}
\nc{\glN}{\mathfrak{gl}_N}
\nc{\glNt}{\mathfrak{gl}_N[t]}
\nc{\s}{sing}

\nc{\ep}{\epsilon}
\nc{\bla}{{\bs\la}}
\nc{\ga}{\gamma}
\nc{\Ga}{\Gamma}
\nc{\Ss}{{\mathcal S}}
\nc{\CC}{{\mathcal C}}
\nc{{\bal}}{{\Ann(\bar A^l)}}
\nc{\Sing}{{\on{Sing}}}
\nc{\id}{\on{id}}
\nc{\sgn}{\on{sgn}}
\nc{\FF}{\es{\on{Flag}}}
\renewcommand{\phi}{\varphi}

\title{The contravariant form on singular vectors of a projective arrangement}

\author[M. Falk]{Michael J. Falk}
\address{Department of Mathematics and Statistics,
Northern Arizona University\\
 Flagstaff, AZ 86011, USA}
\email{\href{mailto:michael.falk@nau.edu}{michael.falk@nau.edu}}
\urladdr{\href{http://www.cefns.nau.edu/~falk/}%
{www.cefns.nau.edu/\char'176falk}}

\author[A. Varchenko]{Alexander N. Varchenko$^*$}
\address{Department of Mathematics, University of
North Carolina at Chapel Hill\\
Chapel Hill, NC 27599, USA}
\email{\href{mailto:anv@email.unc.edu}{anv@email.unc.edu}}
\urladdr{\href{http://www.math.unc.edu/Faculty/av/}
{www.math.unc.edu/Faculty/av/}}
\thanks{{$^*$}Partially supported by NSF grant DMS-1101508}

\begin{document}
\maketitle
\begin{abstract}
We define the flag space and space of singular vectors for an arrangement \A\
of hyperplanes in projective space equipped with a system of weights $a \colon \A \to \C$.
We show that the contravariant bilinear form of the corresponding weighted
central arrangement induces a well-defined form on the space of singular vectors
 of the projectivization.
If $\sum_{H\in\A} a(H)=0$, this form is naturally isomorphic to the restriction
to the space of singular vectors of the
contravariant form of any affine arrangement obtained from \A\ by dehomogenizing
 with respect to one of its hyperplanes.

\end{abstract}
\section{Introduction}
\label{sec:intro}
Let $\A=\{H_1,\ldots, H_n\}$ be an arrangement of affine hyperplanes in $\C^\ell$.
Let $f_i \colon \C^\ell \to \C$ be an affine linear functional with zero locus
 $H_i$, for $1\leq i\leq n$. Let $M=M(\A)=\C^\ell - \bigcup_{i=1}^n H_i$ be the complement
 to the arrangement.
If $W$ is a \C-vector space, then $W^*$
denotes its dual space. Let $\C^\times = \C-\{0\}$.

Let $\omega_i=d \log(f_i)$ for $1\leq i\leq n$. Denote by $A$ the
$\C$-subalgebra of the holomorphic De Rham complex of $M$ generated by the closed forms
 $1, \omega_1, \ldots, \omega_n$. The algebra $A$ is graded,
 $A=\oplus_{p= 0}^\ell A^p$, and called the Arnol'd-Brieskorn-Orlik-Solomon algebra
 or
  the OS algebra of \A.
  The dual space  $\F=\F(\A):=\oplus_{p\geq 0} \, \F^p$ of $A$
  is called the {\em flag space} of \A, \cite{SV91}.

Let $a=(a_1,\ldots,a_n) \in \C^n$ be a vector of weights. The {\em contravariant form} of
 the weighted arrangement $(\A,a)$ is the symmetric bilinear form $S=\oplus S_p \colon \F \otimes
 \F \to \C$, where $S_p \colon \F^p \otimes \F^p \to \C$ is defined by
\begin{equation}\label{eqn:contra}
S_p(F,F')=\sum_J a_J F(\omega_J)F'(\omega_J).
\end{equation}
The sum is over all sequences $J=(j_1,\cdots , j_p)$ with $1\leq j_1< \cdots <j_p\leq n$,
 $a_J=\prod_{i=1}^p a_{j_i}$ and $\omega_J=\omega_{j_1} \we \cdots \we \omega_{j_p}$,  \cite{SV91}.

In particular, if
$\{F_1, \ldots, F_n\} \subseteq \F^1$ is the basis dual to the basis
$\{\omega_1, \ldots, \omega_n\}$ of $A^1 \cong \C^n$, then
\begin{equation}\label{eqn:f1}
S_1(F_i,F_j)=a_i\delta_{ij}.
\end{equation}

The contravariant form  has many remarkable properties, see \cite{SV91,V95, V06, V11}.
It is a generalization of the Shapavalov form associated to a tensor product of
highest weight
representations of a simple Lie algebra -- for this application \A\ is
a discriminantal arrangement and $a$ is determined by the representations.

 The space $\F$ has a combinatorially defined differential $d \colon \F^p \to \F^{p+1}$.
 The space $A$ has a differential $\delta_a \colon A^p \to A^{p+1}$
 defined by multiplication by  $\om_a:=\sum_{i=1}^na_i\om_i$.
 The contravariant form $S$ induces a morphism of complexes $\psi \colon (\F,d) \to (A,\delta_a)$,
 see \cite{SV91} and Section \ref{sec:flag}. The pair $(\F,d)$ is the {\em flag complex} of \A.

Let $\Sing(\F^\ell)=\Sing_a(\F^\ell) \subseteq \F^\ell$ be the annihilator of $\omega_a \we A^{\ell-1}$.
It is called the subspace of {\em singular vectors} of $\F^\ell$, relative to $a$.
 This terminology is introduced in \cite{ V06} and motivated by \cite{SV91}.
 In \cite{SV91} the subspace
 $\Sing(\F^\ell)$ for a discriminantal arrangement is interpreted as the subspace of
 singular vectors of a tensor product of Verma modules over a Kac-Moody algebra.
The inclusion $\Sing(\F^\ell) \hookrightarrow (A^\ell)^*$ induces an isomorphism
\[
\Sing(\F^\ell) \to (H^\ell(A,\delta_a))^*=(A^\ell/(\omega_a \we A^{\ell-1}))^*.
\]

Let $\Phi_a=\prod_i f_i^{\,-a_i}$ be the master function associated with $(\A,a)$, and let $\L_a$ be the
rank-one local system on $M$ whose local sections are the multiples of single-valued branches of $\Phi_a$.
The inclusion of $(A,\delta_{ca})$ into the twisted algebraic de Rham complex of $\L_{ca}$ induces an isomorphism
\[
H^*(A,\delta_{ca}) \cong H^*(M, \L_{ca})
\]
for generic $c$ \cite{SV91}. Since $\Sing_{ca}(\F^\ell)=\Sing_a(\F^\ell)$ for any nonzero $c$,
this implies that $\Sing_a(\F^\ell)$ is isomorphic to the local system homology $H_\ell(M,\L_{-ca})$ for generic $c$.

An important object is the restriction of the contravariant form $S_\ell$ to the subspace
$\Sing(\F^\ell)$. It relates linear and nonlinear characteristics of the weighted arrangement $(\A,a)$.
For example, the rank of the restriction of $S_\ell$ to $\Sing(\F^\ell)$
bounds from above
the number of critical points of the master function $\Phi_a,$
see \cite{V11}. For other properties see in \cite{V06, V11}.

The complement $M$ to $\A$ in $\C^\ell$ can be identified with the complement
to the arrangement of projective hyperplanes $\A_\infty$ in projective
 space $\P^{\ell}$ consisting
of the closures of the hyperplanes of \A\ together with the hyperplane $H_\infty$ at
infinity. The purpose of this note is  to develop the notions of
the flag space \F,
the space $\Sing(\F^\ell)$ of singular vectors of $\F^\ell$, and
the bilinear form $S_\ell\vert_{\Sing(\F^\ell)}$ on
$\Sing(\F^\ell)$, starting with
the projective arrangement $\A_\infty$ in $\P^\ell$. Namely,
 the purpose is to  define these objects in such
a way that they will not depend on the choice of the
hyperplane at infinity of the projective arrangement.

The following example  illuminates our constructions.

\begin{example}
\label{ex:simple}
Let \A\ be the affine arrangement in $\C^1$
 of $n$ distinct points $z_1,\ldots, z_n$. Then
  $\omega_i=d \log(x-z_i)$
   for $1\leq i\leq n$. We have $A^0 \cong \C$, $A^1 \cong \C^n$ and $A^p=0$ for $p\geq 2$.
Let $a\in \C^n$ be a vector of weights. The contravariant form of $(\A,a)$ on $\F^1$ is
 given by \eqref{eqn:f1}. The image $\omega_a \we A^0$
 of $\delta_a$ in $A^1$ is spanned by $\omega_a$, and
  the subspace $\Sing(\F^1)\subset \F^1$ of singular vectors is the subspace
$
\{ \sum_{i=1}^n c_iF_i \in\F^1 \ \mid\ \sum_{i=1}^n c_ia_i=0\}.
$

Our construction identifies the pair $(\Sing(\F^1), S_1\vert_{\Sing(\F^1)})$ with the pair
\bea
\left((\Ann(\tilde{\omega}_a \we \tilde{A}^0)+\Ann(q^*A^1))/\Ann(q^*A^1),\ {}
 \tilde S_1\vert_{(\Ann(\tilde{\omega}_a \we \tilde{A}^0)+\Ann(q^*A^1))/\Ann(q^*A^1)}\right)
\eea
described below.

Let $[u:v]$ be homogeneous coordinates on $\P^1$,
 with $x=\frac{v}{u}.$ The projectivization $\A_\infty$ of \A\
 is the arrangement in
  $\P^1$  of the points
$
p_1=[1:z_1], \ldots,  p_n=[1:z_n]
$
and the point $p_0=[1:0]$ at infinity.
The weight of $p_i$ is $a_i$ for $1\leq i\leq n$ and the weight of $p_0$ is  $a_0=-\sum_{i=1}^n a_i$.

In our construction
we use the associated central arrangement in $\C^2$, the cone $\tilde{\A}$ of $\A_\infty$,
 consisting of the lines $v-z_iu=0$ for $1\leq i\leq n$ and the line $u=0$.
Introduce the following one-forms on $\C^2$:
  $\tilde{\omega}_i=d\log (v-z_iu)$ for $1\leq i\leq n$, and
   $\tilde{\omega}_0=d\log(u)$.
The arrangement $\tilde \A$ is weighted with the weights $\tilde a=(a_0,\dots, a_n)$.
We will denote by $\tilde M, \tilde A, \tilde{\F}, \tilde{\omega}_a, \tilde S$ the complement,
Orlik-Solomon algebra, flag space of $\tilde \A$, special element, and the contravariant form of
$(\tilde \A, \tilde a)$,
respectively.

The orbit map $q \colon \tilde M \to \tilde M/\C^\times = \C - \{z_1, \ldots, z_n\}$
induces an injection $q^*\colon A \to \tilde A$
whose image is
the subalgebra generated by
$
\{\sum_{i=0}^n \la_i \tilde{\omega}_i \in \tilde{A}^1\ \mid\ \sum_{i=0}^n \la_i=0\}\subset \tilde{A}^1.
$
One computes
\bea
 q^*(\omega_i)= q^*(d \log (x-z_i))= d \log(\frac{v}{u} - z_i)=
\tilde \omega_i - \tilde \omega_0.
\eea
Then the special element $\omega_a$ of $A^1$
is mapped by $q^*$
to the special element $\tilde \omega_a=\sum_{i=0}^n a_i\tilde{\omega}_i$ of $\tilde A^1$.
Identifying $A^1$ with $q^*A^1$, the flag space $\F^1=(A^1)^*$
is isomorphic to the quotient of
 $\tilde{\F}^1$ by the annihilator $\Ann(q^*A^1)\subset \tilde F^1$
 of $q^*A^1\subset \tilde A^1$.
The subspace $\Ann(q^*A^1)$ is spanned by $\sum_{i=0}^n \tilde{F}_i$.
(Notice that in this consideration the index 0 does not play any
special role.)
The subspace  $\Ann(\tilde{\omega}_a \we \tilde{A}^0)$ of $\tilde\F^1$ consists of
flags $\sum_{i=0}^n c_i\tilde{F}_i$ such that $\sum_{i=0}^n c_ia_i=0$.
This subspace is orthogonal to the subspace $\Ann(q^*A^1)$ relative to
 the contravariant form of $\tilde{\A}$. Indeed
 we have
$
\tilde S_1(\sum_{i=0}^n \tilde{F}_i \, , \sum_{i=0}^n c_i\tilde{F}_i)=
\sum_{i=0}^n c_ia_i= 0.
$
Thus, the contravariant form $\tilde{S}_1$ induces a well-defined form on
 the image of
 $\Ann(\tilde{\omega}_a \we \tilde{A}^0)$ in $\tilde{\F}^1/\Ann(q^*A^1)$,
 namely, a form on
 \bea
 (\Ann(\tilde{\omega}_a \we \tilde{A}^0)+\Ann(q^*A^1))/\Ann(q^*A^1)
 \cong
 \Ann(\tilde{\omega}_a \we \tilde{A}^0)/(\Ann(\tilde{\omega}_a \we \tilde{A}^0)\cap\Ann(q^*A^1)).
 \eea

 The flags $\tilde{F_1}, \ldots, \tilde{F_n}$ induce
 a basis
 of $\tilde{\F}^1/\Ann(q^*A^1)$. Using this basis, we see that
 the form induced by $\tilde{S}_1$ on  $(\Ann(\tilde{\omega}_a
 \we \tilde{A}^0)+\Ann(q^*A^1))/\Ann(q^*A^1)$
  corresponds to the restriction of the
  original form $S_1$ to the subspace $\Sing(\F^1)$ under the
   isomorphism of $\F^1$ with $\tilde{\F}^1/\Ann(q^*A^1)$.

   Notice that the form $\tilde S_1$ does not induce a well-defined
   form on $\F^1=\tilde \F/\Ann(q^*A^1)$ -- the extension of $S_1\vert_{\Sing(\F^1)}$
   defined by \eqref{eqn:f1} depends on the choice of hyperplane at infinity.

\end{example}

In general, for any  weighted affine   arrangement $(\A,a)$ in $\C^\ell$,
we identify the pair
\linebreak
$(\Sing(\F^\ell), S_\ell\vert_{\Sing(\F^\ell)})$
with the pair
\bea
\left((\Ann(\tilde{\omega}_a \we \tilde{A}^{\ell-1})+\Ann(q^*A^\ell))/\Ann(q^*A^\ell), \tilde S_\ell\vert_{(\Ann(\tilde{\omega}_a \we \tilde{A}^{\ell-1})+\Ann(q^*A^\ell))/\Ann(q^*A^\ell)}\right)
\eea
expressed in terms of the cone $\tilde \A$ of the projectivization $\A_\infty$ of $\A$.

\medskip

Our statement that  the pair $(\Sing(\F^\ell), S_\ell\vert_{\Sing(\F^\ell)})$
can be constructed
in terms of $\A_\infty$,  without choosing a particular hyperplane at infinity,
is analogous to the following fact from representation
theory. Let $V_{\La_i}$, $i=0,\dots,n$, be irreducible finite dimensional
highest weight
representations  of a simple Lie algebra. Here $\La_i$ is the highest weight of $V_{\La_i}$.
Let $\La_0^\vee$ be the highest weight of the representation dual to $V_{\La_0}$.
 Let $S_i$ be the Shapavalov form on $V_{\La_i}$.
Let
$\Sing (\otimes_{i=0}^{n} V_{\La_i})[0]\subset \otimes_{i=0}^{n} V_{\La_i}$
 be the subspace
of singular vectors of weight zero and
$\Sing (\otimes_{i=1}^{n} V_{\La_i})[\La_0^\vee]\subset \otimes_{i=1}^{n} V_{\La_i}$
 the subspace
of singular vectors of weight $\La^\vee_0$. Then the pair
$(\Sing (\otimes_{i=1}^{n} V_{\La_i})[\La_0^\vee], (\otimes_{i=1}^{n} S_i)
\vert_{\Sing (\otimes_{i=1}^{n} V_{\La_i})[\La_0^\vee]})$
is isomorphic to the pair
$(\Sing (\otimes_{i=0}^{n} V_{\La_i})[0], (\otimes_{i=0}^{n} S_i)
\vert_{\Sing (\otimes_{i=1}^{n} V_{\La_i})[0]})$.

\section{Flag complex and contravariant form of a central arrangement}
\label{sec:flag}

We recall in more detail some of the theory of flag complexes from \cite{SV91}.
The following notation, which differs from the notation of \S \ref{sec:intro},
will be used throughout the rest of the paper.
For general background on arrangements see \cite{OT92}.

Suppose $\A=\{H_0,\ldots, H_n\}$ is a
central arrangement in $\C^{\ell+1}$.
Let  $f_0, \ldots, f_n \in (\C^{\ell+1})^*$ with $H_i=\ker(f_i)$ for $0\leq i\leq n$.
Let $\omega_i=\frac{df_i}{f_i}$ for $0\leq i\leq n$, and let $A$ be the
OS algebra of \A, as defined in \S \ref{sec:intro}. Let $E$ be the
graded exterior algebra over \C\ with generators
$e_0, \ldots, e_n$ of degree one. Let $\partial \colon E^p \to E^{p-1}$
be defined by
\[
\partial(e_{j_1} \we \cdots \we e_{j_p})=\sum_{i=1}^p (-1)^{i-1} e_{j_1}
\we \cdots \we \hat{e}_{j_i} \we \cdots \we e_{j_p},
\]
where $\hat{\hspace*{1em}}$ denotes deletion.
If $J=(j_1, \ldots, j_p)$, denote the product $e_{j_1} \we \cdots \we e_{j_p}$
by $e_J$. Say $J$ is {{\em dependent} if $\{f_i \mid i \in J\}$ is linearly
 dependent in $(\C^{\ell+1})^*$. Let $I$ be the ideal of $E$ generated by
$\{\partial e_J \mid J \ \text{is dependent}\}$.  By \cite{OS80}, the
surjection $E \to A$ sending $e_i$ to $\omega_i$ has kernel $I$.
We tacitly identify $A$ with $E/I$. The map $\partial$ induces a well-defined map
  $\partial \colon A \to A$, a graded derivation of degree $-1$, and $(A,\partial)$ is
   a chain complex.

Let $L=L(\A)$ be the intersection lattice of \A, the set of intersections
of subcollections
of \A, partially-ordered by reverse inclusion. Let $\FF=\oplus_{p=0}^{\ell+1} \FF^p$
 be the graded
\C-vector space with $\FF^p$ having basis consisting of chains $(X_0 < \cdots < X_p)$ of $L$
satisfying $\codim(X_i)=i$ for $0\leq i\leq p$.
Such a chain will be called a {\em flag}.
For each ordered subset $J=(j_1, \ldots, j_p)$ of $\{0,\ldots,n\}$, let $\xi(J)$
be the chain
\(
(X_0 < \cdots<  X_p)
\)
of $L$, where $X_0=\C^{\ell+1}$ and $X_i=\bigcap_{k=1}^i H_{j_k}$ for $1\leq i\leq p$. Note that $\xi(J)$ is a flag if and
 only if $\{H_i \mid i \in J\}$ is independent in \A. If $\pi$ is a permutation
 of $\{1, \ldots, p\}$, let $J^\pi=(j_{\pi(1)}, \dots, j_{\pi(p)})$.
For any flag $F \in \FF^p$ and any ordered $p$-subset $J$ of
$\{1, \ldots, n\}$, there is at most one permutation $\pi$ such that  $F=\xi(J^\pi)$.

Define a bilinear pairing
\[
\langle \hspace*{1em}, \hspace*{1em} \rangle \colon \FF^p \otimes E^p \to \C
\]
by
\beq
\label{eqn:pairing}
\langle F, e_J \rangle = \begin{cases}
\sgn(\pi) \ \ \text{if} \ \xi(J^\pi)=F \\
0 \  \ \ \ \ \ \ \ \ \  \text{otherwise}
\end{cases},
\eeq
for every flag $F$ in $\FF^p$ and ordered $p$-subset $J$ of $\{0, \ldots, n\}$.

\begin{prop}[\cite{SV91}]
\label{prop:well} 
$\langle F, \partial e_J\rangle =0$ for every $F \in \FF^p$ and dependent $(p+1)$-tuple $J$. Moreover, if $(X_0 < \cdots < X_{i-1} < X_{i+1} < \cdots X_p)$ is a chain in $L$
 with $\codim(X_j)=j$ then
\begin{equation}\label{eqn:gap}
\left\langle \sum_{X_{i-1}<X<X_{i+1}} (X_0 < \cdots < X_{i-1} <X< X_{i+1}
 < \cdots X_p), \ e_J\right\rangle = 0,
\end{equation}
for every ordered $p$-subset $J$ of $\{0, \ldots, n\}$.
\end{prop}
Let $\F=\oplus_{p=0}^{\ell+1}\, \F^p$ be the quotient of $\FF$ by the (homogeneous) subspace spanned by
the sums
\beq
\sum_{X_{i-1}<X<X_{i+1}} (X_0 < \cdots < X_{i-1} <X< X_{i+1}
 < \cdots X_p)
\label{egn:gap}
\eeq
as $(X_0 < \cdots < X_{i-1} < X_{i+1} < \cdots X_p)$ ranges over all chains in $L$
with $\codim(X_j)=j$.

 Denote the image of $(X_0<\cdots<X_p)$ in $\F^p$ by $[X_0< \cdots <X_p]$.
 By  Proposition~\ref{prop:well},
$\langle \hspace*{1em}, \hspace*{1em} \rangle$ induces a well-defined
 bilinear pairing $
\langle \hspace*{1em}, \hspace*{1em} \rangle \colon \F^p \otimes A^p \to \C$.

The pairing $\langle \hspace*{1em}, \hspace*{1em} \rangle$ is a combinatorial model
of the integration pairing of the ordinary homology and cohomology of the complement $M$ with coefficients in \C,
see \cite{SV91}.

\begin{thm}[\cite{SV91}]
\label{thm:nondeg}
The pairing $\langle \hspace*{1em}, \hspace*{1em} \rangle \colon \F^p
\otimes A^p \to \C$ is nondegenerate.
\end{thm}

Let $\phi \colon A \to \F^*$ be defined by $\phi(x)=\langle -,x \rangle \colon \F \to \C$. By Theorem~\ref{thm:nondeg}, $\phi$ is an isomorphism.
The value of $\phi(\omega_J)$ in terms of the canonical basis of $\FF$
is given in \cite[(2.3.2)]{SV91}.
Similarly, $\phi^* \colon \F \to \A^*$ is an isomorphism, with $\phi^*(F)=\langle F, - \rangle \colon A \to \C$.  \F\ is called the {\em flag space} of \A.

\medskip
Let $d \colon \FF^p \to \FF^{p+1}$ be the linear map defined by
\beq\label{eqn:diff}
d(X_0< \cdots < X_p)=\sum_{\substack{X>X_p \\ \codim(X)=p+1}} (X_0< \cdots < X_p <X).
\eeq
Clearly $d$ induces a linear map $d \colon \F^p \to \F^{p+1}$. Relations
 \eqref{eqn:gap} imply $d\circ d=0$. The pair $(\F,d)$ is called the 
 {\em flag complex} of \A. The following result is a reformulation of
 Lemma~2.3.4 of \cite{SV91}.

\begin{thm}
\label{thm:dual}
For any $F \in \F^p$ and $x \in A^{p+1},$
\[
\langle F, \partial x \rangle = \langle dF, x \rangle.
\]
\end{thm}

Let $d^* \colon (\F^p)^* \to (\F^{p-1})^*$ be the adjoint of $d\colon \F^{p-1} \to \F^p$.
\begin{cor}[{\cite[Lemma 2.3.4]{SV91}}] The map
$\phi \colon (A,\partial) \to (\F^*,d^*)$ given by $\phi(x)=\langle -,x \rangle$ 
 is an isomorphism of chain complexes.
\label{cor:flagisom}
\end{cor}
\noindent
Similarly $\phi^* \colon (\F,d) \to (A^*,\partial^*)$ is an isomorphism of cochain complexes.

There is a decomposition of \F\ dual to the Brieskorn decomposition
 \cite[Lemma 5.91]{OT92} of $A$.
For $X \in L$ let $\F^p_X$ be
the image in $\F^p$ of the subspace of $\FF^p$ spanned by flags that
terminate at $X$. Then by \cite[(2.12)]{SV91},
\beq
\label{eqn:vertex}
\F^p=\bigoplus_{\codim(X)=p} \F^p_X.
\eeq

\medskip
Let $a=(a_0, \ldots, a_n) \in \C^{n+1}$. Let $\omega_a=\sum_{j=0}^n a_j \omega_j$
and $\delta_a \colon A \to A$ with $\delta_a(x)=\omega_a \we x$. Let
\[
S =\oplus S_p \colon \F \otimes \F \to \C
\]
be the
contravariant form of the weighted arrangement $(\A,a)$, as defined in \eqref{eqn:contra}.
$S$ gives rise
to the map $\F \to \F^*$ that sends $F$ to $S(F,-) \colon \F \to \C$. By composing this map with the
isomorphism $\phi^{-1} \colon \F^* \to A$, one obtains a map $\psi \colon \F \to A$,
characterized by the formula
\begin{equation}
\label{eqn:twining}
S_p(F,F')=\langle F, \psi(F') \rangle,
\end{equation}
for all $F, F' \in \F^p$, for each $p$. $\psi$ is called the {\it contravariant map}.

\begin{thm}[{\cite[Lemma 3.2.5]{SV91}}]
The contravariant map $\psi \colon (\F,d) \to (A,\delta_a)$
is a morphism of cochain complexes.
\end{thm}

\begin{cor}
\label{cor:inter}
For every $F \in \F^p$ and $F' \in \F^{p-1}$,
\[
S(F,d F')=\langle F, \omega_a \we \psi(F')\rangle.
\]
\end{cor}

\section{Projective OS algebra and flag space}

Let \A\ be a central arrangement as in \S \ref{sec:flag}. Let $\bAA$ denote the projectivization of
\A, consisting of the projective hyperplanes
\[
\bH_i : = (H_i-\{0\})/\C^\times
\]
in $\P^\ell=(\C^{\ell+1}-\{0\})/\C^\times$, for  $0\leq i \leq n$.
Let $\bar M=\P^\ell - \cup_{i=0}^n\, \bH_i$ be the complement to \bAA\ in $\P^\ell$.

\begin{defn}[{\cite{CDFV10}}] The OS algebra $\bA=A(\bAA)$ of the projective arrangement
 \bAA\ is the kernel of $\partial \colon A \to A$.
\end{defn}

Denote by $\iota: \bA\ \to A$ the natural imbedding.
Let $(\F,d)$ be the flag complex of \A.

\begin{defn} The {\em flag space} $\bF=\F(\bAA)$ of the projective arrangement
 \bAA\ is the quotient $\F/\im(d)$.
\label{def:flag}
\end{defn}

Thus \bF\ is obtained from \F\ by introducing the additional relations
\beq
\sum_{\substack{X>X_p \\ \codim(X)=p+1}} (X_0< \cdots < X_p <X) = 0,
\label{eqn:top}
\eeq
where  $(X_0<\cdots<X_p)$ ranges over all flags of length $p$ in $L$, for $0\leq p\leq \ell$.

Let $\pi \colon \F \to \bF$ be the canonical projection. For $F\in \F$ we write $\bar{F}=\pi(F)$.
Then, for instance, $\sum_{i=0}^n \bar{F}_i=0$, where $\{F_0, \ldots, F_n\} \subseteq \bar{F}^1$
is the basis dual to $\{\omega_0, \ldots, \omega_n\} \subseteq A^1$.

\begin{thm}
\label{prop:isom}
Let $\rho \colon A^* \to \bA^*$ be given by restriction.
Then the isomorphism $\phi^* \colon \F \to A^*$ induces an 
isomorphism $\bar{\phi}^* \colon \bF \to \bA^*$,
given by the commutative diagram
\[
\xymatrix
{
\F \ar[r]^-{\phi^*} \ar[d]_{\pi} &A^*\ar[d]^\rho\\
\bF \ar[r]^-{\bar{\phi}^*} & \bA^*
}
\]

\end{thm}

\noindent
Theorem~\ref{prop:isom} is proved below.

\begin{lem}[{\cite[Lemma 3.13]{OT92}}] The complex
\[
\xymatrix@1{
0 \ar[r] &A^\ell  \ar[r]^-\partial &A^{\ell-1} \ar[r]^-\partial & \ \ \cdots \
\ \ar[r]^-\partial &A^1 \ar[r]^-\partial & A^0 \ar[r] & 0
}
\]
is exact.
\label{lem:partial}
\end{lem}

\begin{cor}\label{cor:flexact}
The complex
\[
\xymatrix@1{
0 \ar[r] &\F^0  \ar[r]^-d &\F^1 \ar[r]^-d & \ \ \cdots \ \
\ar[r]^-d &\F^{\ell-1} \ar[r]^-d &\F^\ell \ar[r] & 0
}
\]
is exact.
\end{cor}
\noindent
Corollary~\ref{cor:flexact}  also follows from \cite[Cor. 2.8]{SV91}.

\begin{prop} $\Ann(\bA) = \im(d \colon \F \to \F)$.
\label{prop:im}
\end{prop}

\begin{proof} By Lemma~\ref{lem:partial}, $\bA=\im(\partial)$.
Then  $F \in \Ann(\bA)$ if and only if $\langle F,\partial x\rangle=0$ for all $x \in A$.
By Theorem~\ref{thm:dual}, this is equivalent to the statement
 $\langle dF,x \rangle=0$ for every $x \in A$, or
  $dF=0$ by Theorem~\ref{thm:nondeg}. Then $\Ann(\bA)=\ker(d)$, which equals $\im(d)$ by Corollary~\ref{cor:flexact}.
\end{proof}

\begin{proofof}{Theorem \ref{prop:isom}} The assertion now
 follows immediately from Proposition~\ref{prop:im} and Definition~\ref{def:flag}.
\end{proofof}

Lemma~\ref{lem:partial} also has the following consequence.
\begin{cor}
\bA\ is the subalgebra of $A$ generated by $1$ and $\bA^1=\{\sum_{i=0}^n c_i \omega_i\ |\
\sum_{i=0}^n c_i=0\}$.
\label{cor:deg1}
\end{cor}

\begin{proof} By Lemma~\ref{lem:partial}, we have
 $\bA=\im(\partial)$. One can show by induction that
\[
\partial \omega_J=(\omega_{j_2}-\omega_{j_1}) \we \cdots \we (\omega_{j_p} - \omega_{j_1}),
\]
for any ordered subset $J=( j_1, \ldots,j_p)$ of $\{0, \ldots, n\}$ with $p\geq 2$.
Each factor on the right-hand side lies in $(\bA)^1$. Since such $\omega_J$ (along with $1$) span $A$, the result follows.
\end{proof}


\begin{rem}
The algebra $\bA$ is naturally isomorphic to $H^*(\bar M,\C)$ and $\iota:\bar A\to A$ is identified with the homomorphism
$q^*:H^*(\bar M,\C)\to H^*(M,\C)$ induced by the orbit map $q: M\to\bar M$. The space
$\bF$ is naturally isomorphic to the homology space $H_*(\bM,\C)$ and the projection $\pi \colon \F \to \bF$ is identified
with the homomorphism $q_*:H^*(M,\C)\to H^*(\bar M,\C)$.
\end{rem}

\section{Singular subspace and contravariant form for projective arrangements}

Let $\A=\{H_0, \ldots, H_n\}$ be a central arrangement in $\C^{\ell+1}$ as above. Let $a=(a_0, \ldots, a_n) \in \C^{n+1}$ and
$\omega_a=\sum_{i=0}^n a_i \omega_i \in A^1$. We identify the flag space $\F=\F(\A)$ with $A^*$ via the map $\phi^*$ of
Section~\ref{sec:flag}.

\begin{defn} The {\em singular subspace} $\Sing(\bF^\ell)\subset \bF^\ell$\ is
\bea
\pi(\Ann(\omega_a \we A^{\ell-1})) = (\Ann(\omega_a \we A^{\ell-1})+\im(d))/\im(d)
\subset \F^\ell/\im(d)=\bar{\F}^\ell.
\eea


\label{def:sing}
\end{defn}

 Let $S \colon \F \otimes \F \to \C$ be the contravariant form of the
central arrangement \A, as defined in \eqref{eqn:contra}.

\begin{thm}
\label{thm:orth}
The subspaces $\Ann(\omega_a \we A)$ and $\im(d)$ of \F\ are orthogonal
with respect to $S$.
\end{thm}

\begin{proof} In \S 2 we constructed the contravariant
map $\psi \colon \F \to A$ satisfying $S_p(F,F')=\langle F, \psi(F') \rangle$ for every $F, F' \in \F^p$.
By Corollary~\ref{cor:inter} $\psi$ satisfies $S_p(F,dF')=\langle F,
\omega_a \we \psi(F') \rangle$ for all $F \in \F^p$ and $F' \in \F^{p-1}$.
Suppose $F \in \Ann(\omega_a \we A^{p-1}) \subseteq \F^p$. Then, for every
 $F' \in \F^{p-1}$, $\langle F, \omega_a \we \psi(F') \rangle=0$. Then
  $S_p(F,dF')=0$ for every $F' \in \F^{p-1}$. Thus $\Ann(\omega_a \we A)$
  is orthogonal to $\im(d)$.
\end{proof}

Define the bilinear form
\[
\bS_\ell \ \colon \ \Sing(\bF^\ell) \otimes \Sing(\bF^\ell)\ \to\ \C
\]
by $\bS_\ell(\bar{F}, \bar{F}')=S_\ell(F,F').$

\begin{cor} The form $\bS_\ell \colon \Sing(\bF^\ell) \otimes \Sing(\bF^\ell) \to \C$ is well-defined.
\end{cor}

\section{Dehomogenization}
\label{sec:decone}
Throughout this section we assume $a \in \C^{n+1}$ satisfies $\sum_{i=0}^n a_i=0.$ Then $\omega_a \in \bA^1$
and $\omega_a \we \bA \subseteq \bA$. Fix a hyperplane $H_j \in \A$. For simplicity
 of notation we assume $j=0$, but the index $0$ will play no special role. Choose coordinates $(x_0, \ldots, x_\ell)$
 on $\C^{\ell+1}$ so that $H_0$ is defined by
the equation $x_0=0$. The {\em decone} of \A\ relative to $H_0$ is an affine arrangement
 $\dA=\{\d H_1, \ldots, \d H_n\}$ in $\C^\ell$. The affine hyperplane
  $\d H_i$ is defined by $\hat{f}_i(x_1, \ldots, x_\ell)=0$, where
  $\hat{f}_i(x_1, \ldots, x_\ell)=f_i(1,x_1, \ldots, x_\ell)$ and
  $f_i \colon \C^{\ell+1} \to \C$ is a linear defining form for $H_i$.
  Let $\hat{M}=\C^\ell - \bigcup_{i=1}^n \d H_i$, \ {} $\hat{\omega}_i=d \log(\hat{f}_i)$, and let $\hat{A}$
  be the algebra of differential forms on $\C^\ell$ generated by 1 
  and $\hat{\omega}_i, 1\leq i \leq n$. Let $\hat{a}=(a_1, \ldots, a_n)$.

\begin{lem}
The map $\epsilon \colon \hat{A} \to \bA$,
$\hat \omega_i \mapsto \omega_i-\omega_0$, is a well-defined isomorphism. Moreover,
$\epsilon$ sends
$\hat{\omega}_{\hat{a}}=\sum_{i=1}^n a_i\hat{\omega}_i$ to $\omega_a$.
\label{lem:eps}
\end{lem}

We note for future reference that
\beq
\label{eqn:hat}
\epsilon(\hat{\omega}_{\hat{J}})=(\omega_{j_1}-\omega_0) \we \cdots
\we (\omega_{j_p}-\omega_0)=\partial \omega_{(0,\hat{J})},
\eeq
for any ordered $p$-subset $\hat{J}=(j_1, \ldots, j_p\}$ of $\{1, \ldots, n\}$, where $(0,\hat{J})=(0,j_1, \ldots, j_p)$.

As in \S\ref{sec:flag}, the flag space $\hat{\F}=\F(\dA)$ of the affine arrangement \dA\
can be identified with $\hat{A}^*$, and the singular
subspace $\Sing(\hat{\F}^\ell)\subset \hat\F^\ell$
relative to $\hat{a}$  is defined by $\Sing(\hat{\F}^\ell)=\Ann(\hat{\omega}_{\hat{a}} \we \hat{A}^{\ell-1})$.
The contravariant form $\hat{S}=\oplus\hat{S}_p$ of \dA\ is given by
\bea
\hat{S}_p(\hat{F},\hat{F'})= \sum_{\hat{J}} \hat{a}_{\hat{J}} \hat{F}
(\hat{\omega}_{\hat{J}})\hat{F'}(\hat{\omega}_{\hat{J}}),
\eea
summing over increasing $p$-tuples $\hat{J}=(j_1,\ldots,j_p)$ of
elements of $\{1,\ldots,  n\}$. We identify \bF\ with $(\bA)^*$ via the isomorphism $\bar{\phi}^*$ of Theorem~\ref{prop:isom}.

In this section we prove the following theorem.

\begin{thm} The map $\epsilon^*\colon \bF\to \hat{\F}$ restricts
 to  an isomorphism of inner-product spaces
\[
\xymatrix@1{
\epsilon^* \colon (\Sing(\bF^\ell),\bS_\ell\vert_
{\Sing(\bF^\ell)})\ar[r]^-\cong & \ (\Sing(\hat{\F}^\ell),\hat{S}_\ell\vert_{\Sing(\hat{\F}^\ell)}).
}
\]
\label{thm:decone}
\end{thm}

Recall that $\bA=\ker(\partial \colon A \to A)$. Define $\sigma \colon \bA^{p-1} \to A^p$ by $\sigma(x)= \omega_0 \we x$.

\begin{lem}
\label{lem:direct}
For each $p,$ we have $
A^p = (\omega_0 \we \bA^{p-1}) \oplus \bA^p.$
\end{lem}
\begin{proof}
By Lemma~\ref{lem:partial}  the complex $(A,\partial)$ is exact.
 Hence, for each $p$, there is a short exact sequence
\bean
\label{A}
\xymatrix@1{
0 \ar[r] &\bA^p  \ar[r]^-\iota &A^p \ar[r]^-\partial & \bA^{p-1} \ar[r] & 0.
}
\eean
This sequence splits: the map $\sigma \colon \bA^{p-1} \to A^p$ defined above
is a section of $\partial \colon A^p \to \bA^{p-1}$.
Indeed, $\partial\circ \sigma(x)=\partial(\omega_0 \we x) =
\partial \omega_0 \we x - \omega_0 \we \partial x = x$ for $x \in \bA^{p-1}$.
Then $A^p=\im(\sigma)  \oplus \ker(\partial)= (\omega_0 \we \bA^{p-1}) \oplus \bA^p$ as claimed.
\end{proof}

Recall that $\ker(\pi)=\Ann(\bA)$. The map $\pi \colon \F^p \to \bF^p$ is
 the adjoint of the inclusion $\bA \to A$. Let $\sigma^* \colon
  \F^p \to \bF^{p-1}$ be the adjoint of $\sigma \colon \bA^{p-1} \to A^p$.

\begin{lem}
\label{lem:finite}
We have the following statements:
\begin{enumerate}
\item  $\F^p=  \ker(\sigma^*) \oplus\ker(\pi)$.

\item The restriction
\(
\pi\vert_{\Ann(\omega_0 \we \bA^{p-1})} \colon \Ann(\omega_0 \we \bA^{p-1}) \to \bF^p
\)
is an isomorphism.
\item $\Ann(\omega_0 \we \bA) = \Ann(\omega_0 \we A)$.
\end{enumerate}
\end{lem}

\begin{proof}
Taking duals in \Ref{A}, we obtain the exact sequence
\beq
\xymatrix{
0 \ar[r] &\bF^{p-1}  \ar[r]^-{\partial^*} &\F^p \ar[r]^-{\pi}
  & \bF^p  \ar[r] & 0.
}
\label{eqn:splexact}
\eeq
The map $\sigma^* \colon \F^p \to \bF^{p-1}$ satisfies
 $\sigma^*\circ \partial^*=(\partial\circ \sigma)^*=\id_{\bF^{p-1}}$.
Then $\F^p = \ker(\sigma^*) \oplus \im(\partial^*)$. 
Statement (i) follows by exactness.

We have $\sigma^*(F)(x)=F(\sigma(x))=F(\omega_0 \we x)$
for all $x \in \bA$. Then $F \in \ker(\sigma^*)$ if and 
only if $F\in \Ann(\omega_0 \we
\bA^{p-1})$, i.e., $\ker(\sigma^*)=\Ann(\omega_0 \we \bA^{p-1})$.
Applying (i), we have $\Ann(\omega_0 \we \bA^{p-1}) \cap \ker(\pi)=0$
and $\bF^p=\pi(\F^p)=\
\pi(\ker(\sigma^*))=\pi(\Ann(\omega_0 \we \bA^{p-1}))$. This proves (ii).

For (iii), assume $F \in \Ann(\omega_0 \we \bA)$ and $x \in 
\omega_0 \we A$. Write $x=\omega_0 \we y$ for $y \in A.$ 
By Lemma~\ref{lem:direct}, we can write $y=y_1+y_2$
with $y_1 \in \omega_0 \we \bA$ and $y_2 \in \bA$.
Then $\omega_0 \we y_1=0$, so $x=\omega_0 \we y_2 \in \omega_0 \we \bA$. Then
$F(x)=0$. Thus $\Ann(\omega_0 \we \bA)
 \subseteq \Ann(\omega_0 \we A)$. The opposite inclusion holds because
 $\omega_0 \we \bA \subseteq \omega_0 \we A$.
\end{proof}

\medskip
Recall the decomposition \eqref{eqn:vertex} of \F. In 
this context Lemma~\ref{lem:finite} yields the following result, which we
will consider in more detail in the next section.
\begin{cor} For $0\leq p\leq \ell$,
\[
\bF^p\cong \bigoplus_{\substack{\codim(X)=p \\ H_0 \not \leq X}} \F^p_X
\]
\label{cor:vertex}
\end{cor}

\begin{proof} Using \eqref{eqn:pairing} one can check easily that
 $[X_0< \cdots< X_p] \in \Ann(\omega_0 \we A^{p-1})$ if and only
 if $H_0\not \leq X_p$. Then, for each $X \in L$ of codimension $p$,
\[
\Ann(\omega_0 \we A^{p-1}) \cap \F^p_X=\begin{cases}0 \ \
\text{if} \ H_0\leq X \\ \F^p_X \ \ \text{if} \ H_0\not\leq X\end{cases}.
\]
The claim then follows from parts (ii) and (iii) of Lemma~\ref{lem:finite}.
\end{proof}

\begin{lem}
We have the following statements:

\begin{enumerate}
\item $\Ann(\omega_a \we \bA) \cap \Ann(\omega_0 \we \bA) \subseteq \Ann(\omega_a \we A)$.
\item $\Ann(\omega_a \we \bA) = \Ann(\omega_a \we A)+\Ann(\bA)$.
\item  $\Sing(\bF^\ell)=\pi(\Ann(\omega_a \we \bA^{\ell-1}))$.
\end{enumerate}

\label{lem:affine}
\end{lem}

\begin{proof}  Let $F \in \Ann(\omega_a \we \bA) \cap \Ann(\omega_0 \we \bA)$ and
 $x\in \omega_a \we A$. Write $x=\omega_a \we y$ with $y \in A$. 
 By Lemma~\ref{lem:direct} we can write $y=y_1+y_2$ with
 $y_1 \in \omega_0 \we \bA$ and $y_2 \in \bA$. Write $y_1 = 
 \omega_0 \we y_1'$ with $y_1'\in \bA$. Since $\omega_a \we \bA \subseteq \bA$,
  $\omega_a \we y_1 =\omega_a \we (\omega_0 \we y_1')= 
  \omega_0 \we (-(\omega_a \we y_1')) \in \omega_0 \we \bA$.
   Then $x=\omega_a \we y = \omega_a \we y_1 + \omega_a \we y_2 \in \omega_0 \we \bA +  \omega_a
  \we \bA$. Then $F(x)=0$. This proves (i).

Let $F \in \Ann(\omega_a \we \bA)$. By part (i) of 
Lemma~\ref{lem:finite}, we can write $F=F_1+F_2$ where
 $F_1 \in \ker(\sigma^*)= \Ann(\omega_0 \we \bA)$
 and $F_2 \in \ker(\pi)=\Ann(\bA)$. Since $\omega_a 
 \we \bA \subseteq \bA$, $F_2 \in \Ann(\omega_a \we \bA)$. Then $F_1=F-F_2
 \in \Ann(\omega_a \we \bA)$. Then
 $ F_1 \in \Ann(\omega_a \we \bA)\cap \Ann(\omega_0 \we \bA).$
 Then $F_1 \in \Ann(\omega_a \we A)$ by (i). Thus $F =F_1+F_2 
 \in \Ann(\omega_a \we A)+\Ann(\bA)$. This proves $\Ann(\omega_a \we \bA)
 \subseteq \Ann(\omega_a \we A)+\Ann(\bA)$. The opposite 
 inclusion follows easily from the fact that
 $\omega_a \we \bA \subseteq (\omega_a \we A) \cap \bA$. This proves (ii).

 Part (iii) follows immediately from (ii).
\end{proof}

We note the following consequence of part (iii) of
 Lemma~\ref{lem:affine} for later use. As observed 
 earlier, $(\bA,\delta_a)$ is a subcomplex  of $(A,\delta_a)$.

\begin{cor}  The inclusion $\Sing(\bF^\ell)
 \hookrightarrow \bF^\ell=(\bA^\ell)^*$ induces an isomorphism
\[
\xymatrix@1{
\Sing(\bF^\ell) \ar[r]^-{\cong} &(H^\ell(\bA,\delta_a))^*.
}
\]
\label{cor:aomoto}
\end{cor}

\begin{proof} Lemma~\ref{lem:partial} implies $\bA^{\ell+1}=0$, so
 $H^\ell(\bA,\delta_a)=\bA^\ell/(\omega_a \we \bA^{\ell-1})$.
 Then $(H^\ell(\bA,\delta_a))^*$ is isomorphic to the annihilator of $\omega_a \we \bA^{\ell-1}$ in $(\bA^\ell)^*$. This annihilator is equal to 
 \[
 (\Ann(\omega_a \we \bA^{\ell-1}) + \Ann(\bA^\ell))/\Ann(\bA^\ell).
 \]
 By Definition~\ref{def:sing}, Proposition~\ref{prop:im}, and Lemma~\ref{lem:affine}(iii), this is equal to $\Sing(\bF^\ell)$.
\end{proof}
\begin{proofof}{Theorem~\ref{thm:decone}} Let $F \in \Ann(\omega_a \we A^{\ell-1})$,
and let $\hat{x}\in \hat{\omega}_a \we \hat{A}^{\ell-1}$. Then
 $\epsilon^*(\bar{F})(\hat{x})=F(\epsilon(\hat{x}))$.
 Since $\epsilon(\hat{\omega}_{\hat{a}})=\omega_a$,
  $\epsilon(\hat{x}) \in \omega_a \we A^{\ell-1}$, so
 $F(\epsilon(\hat{x}))=0$. Then $\epsilon^*(\bar{F})(\hat{x})=0$. Thus $\epsilon^*(\Sing(\bF^\ell))
 \subseteq \Sing(\hat{\F}^\ell)$.

 Conversely, suppose $\hat{F}
 \in\Sing(\hat{\F}^\ell)$. Write $\hat{F}=\epsilon^*(\bar{F})$ with
 $F \in \F^\ell$. Let $x \in \omega_a \we \bA^{\ell-1}$.
  Then $x \in \bA^\ell$, so $x=\epsilon(\hat{x})$ for some
   $\hat{x} \in \hat{A}$. Since $x \in \omega_a \we \bA^{\ell-1}$,
    $\hat{x}\in \hat{\omega}_a \we \hat{A}^{\ell-1}$. Then
$\hat{F}(\hat{x})=0$ by definition of $\Sing(\hat{\F}^\ell)$. 
Then $F(x)=F(\epsilon(\hat{x}))=\epsilon^*(\bar{F})(\hat{x})=\hat{F}(\hat{x})=0$.
This shows that $F \in \Ann(\omega_a \we \bA^{\ell-1})$. 
Then $F \in \Ann(\omega_a \we A^{\ell-1})$ by part (iii) of Lemma~\ref{lem:affine}.
 Then $\bar{F}\in \Sing(\bF^\ell)$ by definition of $\Sing(\bF^\ell)$. Thus $\Sing(\hat{\F}^\ell)
\subseteq \epsilon^*(\Sing(\bF^\ell))$, and $\epsilon^*$ restricts to
    an isomorphism $\Sing(\bF^\ell) \to \Sing(\hat{\F}^\ell)$.

It remains to prove that $\hat{S}_\ell(\epsilon^*(\bar{F}),\epsilon^*(\bar{F}'))=\bar{S}_\ell(\bar{F},\bar{F}')$
for all $\bar{F},\bar{F}' \in \Sing(\bF^\ell)$.  
By \eqref{eqn:hat}, we have
\bea
\hat{S}_\ell(\epsilon^*(\bar{F}),\epsilon^*(\bar{F}')) &=& \sum_{\hat{J}} \hat{a}_{\hat{J}} \epsilon^*(\bar{F})(\hat{\omega}_{\hat{J}})\epsilon^*(\bar{F}')(\hat{\omega}_{\hat{J}})
=
\\
&=&
\sum_{\hat{J}} a_{\hat{J}} F(\epsilon(\hat{\omega}_{\hat{J}}))F'(\epsilon(\hat{\omega}_{\hat{J}}))
= \sum_{\hat{J}} a_{\hat{J}} F(\partial \omega_{(0,\hat{J})})F'(\partial \omega_{(0,\hat{J})}).
\eea
The sum is over increasing $p$-tuples $\hat{J}$ of elements of $\{1, \ldots, n\}$. By parts (ii) and (iii)
of Lemma~\ref{lem:finite}, we may assume that $F,F' \in \Ann(\omega_0 \we A^{\ell-1})$. Since
$
\partial \omega_{(0,\hat{J})}=\omega_{\hat{J}}\, - \, \omega_0 \we \partial \omega_{\hat{J}},
$
 this implies $F(\partial \omega_{(0,\hat{J})}) = F(\omega_{\hat{J}})$ 
 and similarly for $F'.$ Then the last sum above is equal to
$
\sum_{\hat{J}} a_{\hat{J}} F(\omega_{\hat{J}})F'(\omega_{\hat{J}}).
$
This sum is equal to
 $
 \sum_J a_J F(\omega_J)F'(\omega_J),
 $
  summing  now over all increasing $p$-tuples $J$ of elements of $\{0, \ldots, n\}$,
  again because $F, F' \in \Ann(\omega_0 \we A^{\ell-1})$.
   This equals $\bar{S}_\ell(\bar{F},\bar{F'})$ by definition.
\end{proofof}

\medskip

We close this section with a topological remark.
Consider the (multi-valued) master function $\Phi_a=\prod_{i=0}^n f_i^{\,-a_i}$
on $\C^{\ell+1}$. Since $\sum_{i=0}^n a_i=0$, $\Phi_a$ is invariant under the
action of $\C^\times$, hence induces a (multi-valued) master function
$\bar{\Phi}_a$ on \bM. We have $\bar{\Phi}_a=\hat{\Phi}_{\hat{a}}\circ h$
where  $\hat{\Phi}_{\hat{a}}=\prod_{i=1}^n \hat{f}_i^{\,-a_i}$ is
the master function of $(\dA,\hat{a})$ on $\hat{M}$, and
 $h \colon \bM \to \hat{M}$ is the canonical diffeomorphism.
 The associated rank-one local systems $\hat{\L}_{\hat{a}}$ on 
 $\hat M$ and $\bar{\L}_{{a}}$ on $\bar M$ then satisfy
 $h^*\hat{\L}_{\hat{a}}=\bar{\L}_a$.  The inclusion of
 $(\bA,\delta_{ca})$ in the twisted algebraic de Rham complex
 of $\bar{\L}_{ca}$ induces an isomorphism of $H^*(\bA,\delta_{ca})$
  with $H^*(\bM,\bar{\L}_{ca})$ for generic $c$.
As before, $\Sing_a(\bF^\ell)$ is equal to $\Sing_{ca}(\bF^\ell)$
for any nonzero scalar $c$. Then, by Corollary~\ref{cor:aomoto},
 we have the following corollary.

\begin{cor} For generic $c$, the inclusion $\Sing_a(\bF^\ell) 
\hookrightarrow (\bA^\ell)^*$ induces an isomorphism
\[
\xymatrix@1{
\Sing_a(\bF^\ell) \ar[r]^-{\cong} &H_\ell(\bM, \bar{\L}_{-ca}).
}
\]
\end{cor}

This isomorphism does not involve the choice of a hyperplane at infinity.
Thus we have the following commutative diagram of isomorphisms, for
generic $c$, in which the index $0$ again plays no special role:
\[
\xymatrix
{
\Sing_a(\bF^\ell) \ar[r]^-{\epsilon^*} \ar[d] & \Sing_{\hat{a}}(\hat{\F}^\ell) \ar[d] \\
H_\ell(\bM, \bar{\L}_{-ca}) \ar[r]^-{h_\ast} & H_\ell(\hat{M},\hat{\L}_{-c\hat{a}})
}
\]

\section{Transition functions}
\noindent
\medskip
The right-hand side of the formula in Corollary~\ref{cor:vertex} is the
decomposition of the flag space $\hat{\F}^p$ of the decone \dA, see \cite{SV91}.
It can be considered to be the dehomogenization of the projective flag
space $\bF$ relative to $H_0$. The dehomogenizations relative to
different hyperplanes form a set of ``affine charts" for \bF. We compute
the transition functions.

For $0\leq j\leq n$, let $\hat{A}_j, \hat{\F}_j,$ and $\hat{S}^{(j)}$
denote the OS algebra, flag complex, and contravariant form of the affine
arrangement obtained by deconing \A\ with respect to $H_j$. Let
$\epsilon_j \colon \hat{A}_j \to \bA$ be the isomorphism determined
by $\epsilon(\hat{\omega}_k)=\omega_k - \omega_j$, for $0\leq k\leq n$
and $k\neq j$, as in Lemma~\ref{lem:eps}. Let $\epsilon_j^* \colon \bF \to
\hat{\F}_j$ be the adjoint of $\epsilon_j$. For $0\leq i<j\leq n$, set
 $\tau_{ij} = \epsilon_j^*\circ (\epsilon_i^*)^{-1}$. Then
 $\tau_{ij} \colon \hat{\F}_i \to \hat{\F_j}$
is an isomorphism. Theorem~\ref{thm:decone} has the following corollary.

\begin{cor} The restriction of $\tau_{ij}$ is an isomorphism of
 inner product spaces
\[
\xymatrix@1{
\tau_{ij} \colon (\Sing(\hat{\F}_i^\ell),\hat{S}^{(i)}_\ell\vert_{\Sing(\hat{\F}_i^\ell)})\ar[r]^-\cong & \ (\Sing(\hat{\F}_j^\ell),\hat{S}^{(j)}_\ell\vert_{\Sing(\hat{\F}_j^\ell)}).
}
\]
\label{cor:affisom}
\end{cor}

\medskip
According to Corollary~\ref{cor:vertex}, $\tau_{ij}$ can be
considered to be an isomorphism
\[
\tau_{ij} \colon \bigoplus_{\substack{\codim(X)=p \\ H_i \not
\leq X}} \F^p_X \ \ \longrightarrow \bigoplus_{\substack{\codim(X)=p \\ H_j \not \leq X}} \F^p_X.
\]
We describe this map explicitly.

In the special case $p=1$ there is an easy formula for
$\tau_{ij}$. Let $\{F_0, \ldots, F_n\}$ be the canonical
basis of $\F^1$, and suppose $k \neq i$. Then

\[
\tau_{ij}(F_k)=\begin{cases} F_k \ \ \text{if} \ \ k \neq j\\
-\sum_{r\neq j} F_r \ \ \text{if} \ \ k = j.
\end{cases}.
\]
To describe the general formula, we will use the following lemma.

\begin{lem} Let $X \in L$ with $\codim(X)=p$, and let $H \in \A$.
 Then $\F^p_X$ is spanned by elements $[X_0< \cdots < X_{p-1}<X]$ 
  satisfying $H\not\leq X_{p-1}$.
\label{lem:span}
\end{lem}

\begin{proof} We induct on $p$, the case $p=0$ being trivial. 
Let $p>0$ and $[X_0< \cdots <X_{p-1}<X] \in \F^p$. By the
 inductive hypothesis, we may assume $H\not \leq X_{p-2}$. (Here we rely on the fact that the assignment $[X_0< \cdots < X_{p-1}] \mapsto [X_0 < \cdots < X_{p-1}<X_p]$ determines a well-defined linear map $\F_{X_{p-1}}^{p-1} \to \F_{X_p}^p$.) If $H \not \leq X_{p-1}$ we are done.
  Otherwise, by \eqref{eqn:gap}, we have
\[
[X_0< \cdots <X_{p-1}<X]=\sum_{\substack{X_{p-2}<X'<X \\
 X'\neq X_{p-1}}} - [X_0 < \cdots X_{p-2}< X'< X].
\]
Since $H\not \leq X_{p-2}$ and $H\leq X_{p-1}$, $H\not \leq X'$ 
for any $X'\neq X$ satisfying $X_{p-2}<X'$ and $\codim(X')=p-1$. 
Then every flag $(X_0 < \cdots X_{p-2}< X'< X)$ that appears 
on the right-hand side satisfies the required condition. 
This completes the inductive step.
\end{proof}

\begin{thm} Let $[X_0< \cdots < X_p] \in \F^p$ with
 $H_i\not \leq X_p$. If $ \ H_j \not \leq X_p$, then
\bea
\tau_{ij}([X_0< \cdots < X_p]) = [X_0< \cdots < X_p].
\eea
 If $ H_j \leq X_p$ and  $H_j \not \leq X_{p-1}$, then 
\bea
\tau_{ij}([X_0< \cdots < X_p])= \sum_{X_{p-1}<X', X'\neq X_p} - [X_0 < \cdots < X_{p-1}< X'].
\eea

\end{thm}

\begin{proof} By definition, $\tau_{ij}([X_0< \cdots < X_p])$ 
is the unique element of  $\bigoplus_{\substack{\codim(X)=p \\ 
H_j \not \leq X}} \F^p_X$ that represents the same element of 
$\bF^p$ as $[X_0< \cdots < X_p]$. By \eqref{eqn:top}, the 
right-hand side represents the same element of $\bF^p$ as 
$[X_0< \cdots < X_p]$, in either case. An argument similar 
to the one used in the preceding lemma shows that the 
right-hand side lies in
\(
\bigoplus_{\substack{\codim(X)=p \\ H_j \not \leq X}} \F^p_X
\)
in either case. The claim follows.
\end{proof}

By Lemma~\ref{lem:span}, this theorem is sufficient to 
determine $\tau_{ij}$ uniquely. By Corollary~\ref{cor:affisom}, 
$\tau_{ij}$ sends singular vectors of $\hat{\F}_i^\ell$ to 
singular vectors of  $\hat{\F}_j^\ell$, and preserves the value 
of the contravariant form on such vectors.

\medskip
Similarly, there is an algebra isomorphism $\tau_{ji}^* \colon
\hat{A}_i \to \hat{A}_j$ determined by
\[
\tau_{ji}^*(\hat{\omega}_k)=\begin{cases} \hat{\omega}_k -
\hat{\omega}_i \ \ \text{if} \ k \neq i \\
-\hat{\omega}_i \ \ \text{if} \ k=i.
\end{cases}
\]
As in \S \ref{sec:flag}, there is an isomorphism $\hat{\F}_i^*
\to \hat{A}_i$ defined by the affine version of \eqref{eqn:pairing},
and the contravariant map $\psi_i \colon \hat{\F}_i \to \hat{A}_i$
characterized by the formula
\[
S(\hat{F},\hat{F}')=\langle \hat{F}, \psi_i(\hat{F}') \rangle.
\]
The image of $\psi_i$ is the {\em complex of flag forms} of
 $\hat{\A}_i$. (It is a subcomplex of $(\hat{A}_i,\delta_{\hat{a}_i}).$)
Theorem~\ref{thm:decone} has the following consequence.
\begin{cor} The following diagram commutes:
\[
\xymatrix@=15pt{
\Sing(\hat{F}_i^\ell) \ar[r]^-{\psi_i} \ar[d]_{\tau_{ij}}&
\hat{A}_i^\ell \ar[d]^{\tau_{ji}^*}\\
\Sing(\hat{\F}_j^\ell) \ar[r]^-{\psi_j} & \hat{A}_j^\ell
}
\]
\end{cor}

\begin{ack} This research was started during the intensive research period
``Configuration Spaces: Geometry, Combinatorics and Topology,'' May-June, 2010,
at the Centro di Ricerca Matematica Ennio De Giorgi in Pisa. The authors thank the
institute, organizers, and staff for their hospitality and financial support.
We are also grateful to Sergey Yuzvinsky, who took part in our initial
discussions, for his helpful remarks. The second author thanks for
hospitality IHES where this paper was finished.
\end{ack}

\newcommand{\etalchar}[1]{$^{#1}$}

\end{document}